\documentclass[12 pt]{article}
\usepackage{amssymb}
\usepackage{amsmath}
\allowdisplaybreaks
\usepackage{amsthm}
\usepackage{amscd}
\usepackage{graphicx}
\usepackage{hyperref}
\usepackage{pxfonts}

\renewcommand{\theequation}{\arabic{section}.\arabic{equation}}
\def\vbar{\mathchoice{\vrule height6.3ptdepth-.5ptwidth.8pt\kern- .8pt}
{\vrule height6.3ptdepth-.5ptwidth.8pt\kern-.8pt} {\vrule
height4.1ptdepth-.35ptwidth.6pt\kern-.6pt} {\vrule
height3.1ptdepth-.25ptwidth.5pt\kern-.5pt}}
\def\<{\langle}
\def\>{\rangle}
\def\a{\alpha}
\def\b{\beta}

\newtheorem{df}{Definition}[section]
\newtheorem{thm}{Theorem}[section]
\newtheorem{cor}{Corollary}[section]
\newtheorem{rem}{Remark}[section]

\newtheorem{prop}{Proposition}[section]
\newtheorem{exa}{Example}[section]
\newtheorem{lem}{Lemma}[section]

\date{}
\begin{document}

\title{Hom-Jordan-Malcev-Poisson algebras}
\author{
{ Taoufik  Chtioui$^{1}$
 \footnote {  E-mail: chtioui.taoufik@yahoo.fr}
, Sami  Mabrouk$^{2}$
 \footnote {  E-mail: mabrouksami00@yahoo.fr}
, Abdenacer  Makhlouf$^{3}$
 \footnote {  E-mail: Abdenacer.Makhlouf@uha.fr}
}\\
{\small 1.  University of Sfax, Faculty of Sciences Sfax,  BP
1171, 3038 Sfax, Tunisia} \\
{\small 2.  University of Gafsa, Faculty of Sciences Gafsa, 2112 Gafsa, Tunisia}\\
{\small 3.~ Universit\'e de Haute Alsace, IRIMAS - D\'epartement de Math\'ematiques,
F-68093 Mulhouse, France}}
 \maketitle
\begin{abstract}
The purpose of this paper is to provide and study a Hom-type generalization of Jordan-Malcev-Poisson algebras, called  Hom-Jordan-Malcev-Poisson algebras.  We show that they are closed under twisting by suitable self-maps and   give a characterization of admissible Hom-Jordan-Malcev-Poisson algebras. In addition, we introduce the notion of pseudo-Euclidian Hom-Jordan-Malcev-Poisson algebras and describe its $T^*$-extension. Finally, we generalize the notion of Lie-Jordan-Poisson triple system to the Hom setting and establish its relationships with Hom-Jordan-Malcev-Poisson algebras.
\end{abstract}
{\bf Key words}: Hom-JMP-algebra, Hom-Malcev algebra, Hom-Jordan algebra,  Hom-flexible algebras, admissible Hom-JMP algebra, Pseudo-Euclidian Hom-JMP algebra, Hom-Lie-Jordan-Poisson triple system.
 \normalsize\vskip0.5 cm

%----------------------------------------------------------------------------------------------------------
\section*{Introduction}
\renewcommand{\theequation}{\thesection.\arabic{equation}}
Non-associative algebras become important research field due to their importance  in  various problems
related to  physics and other branches of mathematics.   The first instances of
non-associative Hom-algebras appeared in the study of  quasi-deformations of Lie algebras of
vector fields.
Hom-Lie algebras were first introduced by Hartwig, Larsson and Silvestrov in
order to describe $q$-deformations of Witt and Virasoro algebras using $\sigma$-derivations
(see \cite{hls} ). The corresponding associative type objects, called Hom-associative
algebras were introduced by Makhlouf and Silvestrov in \cite{ms}. Hom-alternative, Hom-Jordan  and Hom-flexible  algebras  were introduced first in \cite{makh} and then considered as sa well as
 Hom-Malcev algebras  in \cite{yau}.

Poisson algebras form an important class of non-associative algebras. They are used in many fields of mathematics and physics.  They play a fundamental  role in Poisson geometry, quantum groups, deformation theory, Hamiltonian mechanics and topological field theories.
Poisson algebras are generalized in many ways. If we omit the commutativity of the associative structure, we get the class of non-commutative Poisson algebras.
Another way to generalize this class is to replace the  associative structure by a Jordan product and the Lie-bracket by a Malcev one. Hence, we obtain a new class of algebras  called Jordan-Malcev-Poisson algebras (JMP-algebras)  which  are defined by a  triple $A$  $(A,[,],\circ)$ consisting of a linear equipped with a Malcev bracket and a Jordan structure  satisfying the Leibniz rule:
$$[x,y\circ z]=[x,y]\circ z+y \circ [x,z].$$
They were introduced by Ait BenHaddou, Benayadi and Boulmane  in \cite{benayadi}. Such algebras can be described  in terms of a single bilinear operation, called admissible JMP algebras. This class contains alternative algebras.
In the particular case where $(A,\circ)$ is associative commutative, $(A,[,],\circ)$ becomes a Malcev Poisson algebra, a concept  introduced first   by Shestakov in \cite{shestakov}.

The purpose of this paper is to study a twisted generalization of JMP algebras, called Hom-JMP algebras and some other algebraic structures (admissible Hom-JMP algebras).  Next,  we introduce and study pseudo-Euclidian Hom-JMP algebras, which are Hom-JMP algebras endowed with symmetric invariant non-degenerate bilinear forms.
We provide a twist construction and extend the $T^*$-extension theory to this class of non-associative Hom-algebras.
We also construct a generalized Hom-triple systems called Hom-Lie-Jordan-Poisson triple systems from  admissible Hom-JMP algebras.

This paper is organized as follows:  In section 1, we summarize the definitions and some key constructions of Hom-JMP algebras. In section 2, we study and highlight relationships between Hom-JMP algebras and admissible Hom-JMP algebras.  In addition, it is shown that admissible Hom-JMP algebras are power Hom-associative.
In Section 3,  we introduce the notion of pseudo-Euclidian Hom-JMP algebras and describe its $T^*$-extension.
Section 4 is devoted to the study of  a Hom version of Lie-Jordan-Poisson triple system algebras and provide its connection with admissible Hom-JMP algebras.

%%%%%%%%%%%%%%%%%%%%%%%%%%%%%%%%%%%%%%%%%%%%%%%%%%
\section{Definitions and Preliminary Results}
\label{sec:basic}
%%%%%%%%%%%%%%%%%%%%%%%%%%%%%%%%%%%%%%%%%%%%%%%%%%%%%%%%%
\subsection{Basic definitions}
In this section, we introduce Hom-Jordan-Malcev-Poisson algebras (Hom-JMP algebras) as a generalization of both Hom-Poisson algebras, Malcev-Poisson algebras and Jordan-Malcev-Poisson algebras (JMP algebras). We show  that Hom-JMP algebras are closed under suitable twisting by weak morphisms.\\
Let us begin with the basic definitions regarding Hom-algebras.
We work over a fixed commutative field $\mathbb{K}$ of characteristic $0$.

\begin{df}
\label{def:homass}
Let $(A,\mu,\alpha)$ be a Hom-algebra.
\begin{enumerate}
\item
The Hom-associator $as_A \colon A^{\otimes 3} \to A$ is defined as
\begin{equation}
\label{homassociator}
as_A (x,y,z)= \mu (\mu(x,y),\a(z))-\mu(\a(x),\mu(y,z)).
\end{equation}
\item
The Hom-algebra $A$ is called a Hom-Jordan algebra if it is commutative and satisfies the Hom-Jordan identity
\begin{equation}
\label{homassociativity}
as_A(x^2,\a(y),\a(x))=0.
\end{equation}
\item
The Hom-Jacobiator $J_A \colon A^{\otimes 3} \to A$ is defined as
\begin{equation}
\label{homjacobian}
J_A(x,y,z) =\circlearrowleft_{x,y,z} \mu (\mu(x,y),\a(z)),
\end{equation} where $\circlearrowleft_{x,y,z}$ denotes the cyclic summation over $x,y,z$.
\item
A Hom-Malcev algebra is a Hom-algebra $(A,[,],\alpha)$ such that $[,]$ is skewsymmetric and that the Hom-Malcev identity
\begin{equation}
\label{homMalcevid}
J_A(\alpha(x),\alpha(y),[x,z]) = [J_A(x,y,z),\alpha^2(x)]
\end{equation}
is satisfied for all $x,y,z \in A$.
\item A Hom-flexible algebra is a Hom-algebra $(A,\mu,\a)$ satisfying:
$$as_A(x,y,z)+as_A(z,y,x)=0,$$
or equivalently, $as_A(x,y,x)=0$, for all $x,y,z  \in A$.
\item A Hom-alternative algebra is a Hom-algebra $(A,\mu,\a)$ satisfying
\begin{align*}
as_A(x,y,z)+as_A(y,x,z)=as_A(x,y,z)+as_A(x,z,y)=0,
\end{align*}
or equivalently, $as_A(x,x,y)=as_A(x,y,y)=0$, for all $x,y,z \in A$.
\end{enumerate}
\end{df}
Note that any Hom-alternative algebra is Hom-flexible.\\
Let $(A,\mu,\a)$ be a Hom-algebra. Define the cyclic Hom-associator $S_A$ by:
$$S_A(x,y,z)=\circlearrowleft_{x,y,z}as_A(x,y,z).$$

Let us recall the definition of a JMP algebra \cite{kubo}.
\begin{df}
\label{def:noncommpoisson}
A JMP algebra $(A,\{,\},\circ)$ consists of
a Malcev algebra $(A,\{,\})$ and
a Jordan  algebra $(A,\circ)$
such that the Leibniz identity
\[
\{x,y\circ z\} = \{x,y\}\circ z + y\circ\{x,z\}
\]
is satisfied for all $x,y,z \in A$.
\end{df}
In a JMP algebra $(A,\{,\},\circ)$, the bracket $\{,\}$ is called the Poisson bracket, and $\circ$ is called the Jordan product.  The Leibniz identity says that $\{x,-\}$ is a derivation with respect to the Jordan product.

Hom-Poisson algebras were first introduced in \cite{ms2} by Makhlouf and Silvestrov.  We now define the Hom-type generalization of a JMP algebra.
\begin{df}
\label{def:hompoisson}
A Hom-JMP algebra $(A,\{,\},\circ,\alpha)$ consists of
a Hom-Malcev algebra $(A,\{,\},\alpha)$ and
a Hom-Jordan algebra $(A,\circ,\alpha)$
such that the Hom-Leibniz identity
\begin{equation}
\label{homleibniz}
\{\alpha(x),y\circ z\} = \{x,y\}\circ \alpha(z) + \alpha(y)\circ\{x,z\}
\end{equation}
is satisfied for any $x,y,z\in A$.
\end{df}
In a  Hom-JMP algebra $(A,\{,\},\circ,\alpha)$, the operations $\{,\}$ and $\circ$ are called the Hom-Poisson bracket and the Hom-Jordan product, respectively.
By the skewsymmetry of the Hom-Poisson bracket $\{,\}$, the Hom-Leibniz identity is equivalent to
\begin{equation}
\label{homleibniz'}
\{x\circ y,\alpha(z)\} = \{x,z\}\circ\alpha(y) + \alpha(x)\circ\{y,z\}.
\end{equation}

A JMP algebra is exactly a Hom-JMP algebra with identity twisting map.

Let $(A,\{,\}_A,\circ_A,\a_A)$ and $(B,\{,\}_B,\circ_B,\a_B)$ be two Hom-JMP algebras.  A weak morphism $f: A \to B$ is a linear map such that:
 $$f \{,\}_A=\{,\}_Bf^{\otimes 2}\ \ \textrm{and}\ \ f \circ_A=\circ_B f^{\otimes 2}.$$
A morphism $f: A \to B$ is a weak morphism such that $f \a_A=\a_B f$.

Note that a quadruple $(A,\{,\},\circ,\a)$ is said multiplicative if and only if the twisting map $\alpha: A \to A$ is a morphism.

The following result says that Hom-JMP algebras are closed under twisting by weak self-morphisms.

\begin{thm}
\label{thm:twist}
Let $(A,\{,\},\circ,\a)$ be a Hom-JMP algebra and $\b \colon A \to A$ be a weak morphism.  Then
$$A_\b = (A,\{,\}_\b = \b\{,\},\circ_\b = \b\circ,\b \a)$$
is also a Hom-JMP algebra.  Moreover, if $A$ is multiplicative and $\beta$ is a morphism, then $A_\beta$ is a multiplicative  Hom-JMP algebra.
\end{thm}

\begin{proof}
In \cite{yau} the author  proved that $(A,\{,\}_\b,\b \a)$ is a Hom-Malcev algebra and $(A,\circ_\b,\b \a)$ is Hom-Jordan algebra.\\
It remains to show the  Hom-Leibniz identity.  Let $x,y,z \in A$, we know that
 \begin{align*}
    \{\alpha(x),y\circ z\} = \{x,y\}\circ \alpha(z) + \alpha(y)\circ\{x,z\}.
 \end{align*}
 Now applying $\b^2$ to the previous identity, we obtain
  \begin{align*}
    \{\b^2\alpha(x),\b^2(y)\circ \b^2(z)\} = \{\b^2(x),\b^2(y)\}\circ \b^2\alpha(z) + \b^2 \alpha(y)\circ\{\b^2(x),\b^2(z)\}.
 \end{align*}
 That is
  \begin{align*}
    \{\b\alpha(x),y\circ_\b z\}_\b = \{x,y\}_\b\circ_\b \b\alpha(z) + \b\alpha(y)\circ_\b\{x,z\}_\b.
 \end{align*}
 Therefore, $A_\b$ is a Hom-JMP algebra.
\end{proof}

 \begin{cor}\label{coro twist}
 Let $(A,\{,\},\circ)$ be a JMP algebra and $\a: A \rightarrow A$ be  a JMP morphism. Then \\
$(A,\{,\}_\a = \a\{,\},\circ_\a = \a \; \circ, \a)$ is a multiplicative  Hom-JMP algebra.
 \end{cor}

%%%%%%%%%%%%%%%%%%%%%%%%%%%%%%%%%%%%%%%%%%%%%%%%%%
\subsection{Admissible Hom-JMP algebras}
%%%%%%%%%%%%%%%%%%%%%%%%%%%%%%%%%%%%%%%%%%%%

Let $(A,\cdot,\alpha)$ be a  Hom-algebra. One can define the two following
new products:
$$[x,y]=x\cdot y-y\cdot x\ \ \textrm{and}\ \ x\circ y=\frac{1}{2}(x\cdot y+y\cdot x) ,\ \textrm{for\ all},\ x,y,z\in A.$$
We will denote $A^-$, (respectively $A^+$) the algebra $A$ with multiplication $[-,-]$, (respectively $\circ$).

\begin{lem}\cite{yau}\label{cyclic ass}
Let $(A,\cdot,\a)$ be Hom-flexible algebra. Then we have
\begin{equation}
    2S_A=J_{A^-}.
\end{equation}
\end{lem}
\begin{lem}\cite{ms}
A Hom-algebra $(A,\cdot,\alpha)$ is flexible if and only if
  \begin{equation}
[\alpha(x),y\circ z] = [x,y]\circ \alpha(z) + \alpha(y)\circ[x,z].
\end{equation}
\end{lem}

\begin{lem}\label{felxible accob}
Let $(A,\cdot,\a)$ be a Hom-flexible algebra. Then
\begin{equation}
    J_{A^-}(x^2,\a(y),\a(x))=0,\ \ \forall x,y \in A,\ \textrm{where}\ x^2=x\cdot x=x\circ x.
\end{equation}
\begin{proof}
Let $x, y \in A$. Then we have
\begin{align*}
 &[[x^2,\a(y)],\a^2(x)]+ [[\a(y),\a(x)],\a(x^2)]+\underbrace{[[\a(x),x^2],\a^2(y)] }_{=0}\\
 &=2[\a^2(x),[y,x]\circ \a(x)]+ 2 [[y,x],\a(x)] \circ \a^2(x)\\
 &=2\bigg \{\a^2(x).([y,x].\a(x))+\a^2(x).(\a(x).[y,x])-([y,x].\a(x)).\a^2(x)-(\a(x).[y,x]).\a^2(x)\\
 &+ ([y,x].\a(x)).\a^2(x)-(\a(x).[y,x]).\a^2(x)+\a^2(x).([y,x].\a(x))-\a^2(x).(\a(x).[y,x]) \bigg \}\\
 &=0.
\end{align*}

\end{proof}
\end{lem}
The following result gives a characterization of Hom-flexible algebras.
\begin{prop}\label{flexible relation}
 A  Hom-algebra $(A,\cdot,\alpha)$ is flexible if and only if
  \begin{equation}\label{AdmissibFlexibleCond}
  as_A(x,y,z)=\frac{1}{4}J_{A^-}(x,y,z)+\frac{1}{4}[\alpha(y),[z,x]]+as_{A^+}(x,y,z),\ \textrm{for\ all},\ x,y,z\in A.
  \end{equation}
\end{prop}
\begin{proof}
If $(A,\cdot,\a)$ is Hom-flexible, then by Lemma \ref{cyclic ass}, we have
\begin{align*}
    &J_{A^-}(x,y,z)+[\a(y),[z,x]]+4as_{A^+}(x,y,z)\\
  &=2S_A(x,y,z)+\a(y)(zx)-\a(y)(xz)-(zx)\a(y)+(xz)\a(y)+(xy)\a(z) \\
&  +(yx)\a(z)+ \a(z)(xy)+\a(z)(yx)  -\a(x)(yz)-\a(x)(zy)-(yz)\a(x)-(zy)\a(x) \\
&=2S_A(x,y,z)-as_A(y,z,x)+as_A(y,x,z)-as_A(z,x,y)+as_A(x,z,y)+ as_A(x,y,z)-as_A(z,y,x)\\
&= 4 as_A(x,y,z).
\end{align*}
Conversely, suppose that Eq. \eqref{AdmissibFlexibleCond}  holds, then
\begin{align*}
&  as_A(x,y,x)=\frac{1}{4}J_{A^-}(x,y,x)+\frac{1}{4}[\alpha(y),[x,x]]+as_{A^+}(x,y,x)\\
&= \frac{1}{4} \Big([[x,y],\a(x)]+[[y,x],\a(x)]+[[x,x],\a(y)]\Big) +(x\circ y)\circ \a(x)-\a(x) \circ (y \circ x) \\
&=0. \ \   (\textrm{Since} \circ \textrm{is commutative}).
\end{align*}
Hence $A$ is Hom-flexible.
\end{proof}

\begin{cor}\label{equality of associator}
Let $(A,\cdot,\a)$ be a Hom-flexible algebra. Then
$$as_A(x^2,\a(y),\a(x))= as_{A^+}(x^2,\a(y),\a(x)).$$
\end{cor}
\begin{proof}
Straightforward.
\end{proof}
\begin{df}
  A Hom-algebra $(A,\cdot,\alpha)$ is said to be  an admissible Hom-JMP algebra if $(A,[,],\circ,\alpha)$ is a Hom-JMP algebra.
\end{df}
 \begin{rem}
 Given a Hom-JMP algebra $(A,\{,\},\circ,\a)$, then the vector space $A$ endowed with the morphism $\a$ and the product  defined by
 $$x\cdot y:=\frac{1}{2}\{x,y\}+x \circ y,$$
is  an admissible Hom-JMP algebra.
 \end{rem}

\begin{prop}\label{twist admissibleJMP}
  \label{thm:twist}
Let  $(A,\cdot,\alpha)$ be an admissible Hom-JMP algebra and $\beta \colon A \to A$ be a weak morphism.  Then
\[
A_\beta =  (A,\cdot_\beta,\beta\alpha)
\]
is also an admissible Hom-JMP algebra, where $x\cdot_\beta y=\beta(x)\cdot\beta (y)$.  Moreover, if $A$ is multiplicative and $\beta$ is a morphism, then $A_\beta$ is a multiplicative admissible Hom-JMP algebra.
 \end{prop}

\begin{proof}
straightforward.
\end{proof}
\begin{exa}
Every Hom-alternative algebra is an admissible Hom-JMP algebra.
\end{exa}
\begin{rem}
Note that not all admissible Hom-JMP algebras are hom-alternative algebras. Indeed,
let $A$ be the three-dimensional   algebra of basis $\{e_1,e_2,e_3\}$ defined by:
\begin{center}
\begin{tabular}{c|c|c|c}
  % after \\: \hline or \cline{col1-col2} \cline{col3-col4} ...
  $\star$ & $e_1$ & $e_2$ & $e_3$ \\
   \hline
  $e_1$ & $0$ & $e_2$ & $-e_3$ \\
  \hline
  $e_2$ & $-e_2$ & $0$ & $e_1$ \\
  \hline
  $e_3$ & $e_3$ & $-e_1$ & $0$
\end{tabular}
\end{center}
According to \cite{benayadi}, $A$ is an admissible JMP algebra. Let  the morphism $\a: A \to A$ defined by:
$$\a(e_1)=e_1,\ \a(e_2)=\lambda e_2,\ \a(e_3)=\frac{1}{\lambda} e_3,\ \ \lambda \in \mathbb{K}\backslash\{0\}.$$
Then, in view of  Proposition \ref{twist admissibleJMP},  $(A,\star_\a,\a)$ is an admissible Hom-JMP algebra.
On the other hand, $$as_A(e_2,e_3,e_3)=(e_2\star_\a \star e_3)\star_\a \a(e_3)-\a(e_2)\star_\a(e_3 \star_\a e_3)=-\frac{1}{\lambda^2}e_3 \neq 0.$$ Then $A$ is not a Hom-alternative algebra.

\end{rem}

\begin{rem}\label{admiss==>flexible}
  Every admissible Hom-JMP algebra is Hom-flexible.
\end{rem}
\begin{thm}\label{admissible JMP alg}
Let $(A,\cdot,\a)$ be a Hom-flexible and Hom-Malcev admissible algebra. Then $A$ is an admissible Hom-JMP algebra if and only if $(A,\cdot,\a)$ satisfies the identity
\begin{align}\label{admissible JMP alg cond}
    R_{\a^2(x)}L_{x^2}\a=L_{\a(x^2)}R_{\a(x)}\a,\ \ \ \forall x \in A,
\end{align}
where $L_x$ (respectively, $R_x$) is the left multiplication (respectively, the right multiplication) by $x$ in the algebra $(A,\cdot,\a)$.
\end{thm}
 \begin{proof}
Since $(A,\cdot,\a)$ is Hom-flexible and due to Corollary \ref{equality of associator}, we have
$$as_A(x^2,\a(y),\a(x))= as_{A^+}(x^2,\a(y),\a(x)).$$
The rest of the proof follows immediately.
 \end{proof}

\begin{exa}
Consider the five-dimensional Hom-algebra $(A,\cdot,\a)$ with basis $\{e_1,\ldots,e_5\}$ and multiplication table:
\begin{center}
\begin{tabular}{c|ccccc}
$\cdot$ & $e_1$ & $e_2$ & $e_3$ & $e_4$ & $e_5$\\
\hline
$e_1$ & $0$ & $e_5 + \frac{1}{2}e_4$ & $0$ & $\frac{\nu}{2}e_1$ & $0$\\
$e_2$ & $e_5 - \frac{1}{2}e_4$ & $0$ & $0$ & $-\frac{\nu^{-1}}{2}e_2$ & $0$\\
$e_3$ & $0$ & $0$ & $0$ & $\frac{\lambda}{2}e_3$ & $0$\\
$e_4$ & $-\frac{\nu}{2}e_1$ & $\frac{\nu^{-1}}{2}e_2$ & $-\frac{\lambda}{2}e_3$ & $-e_5$ & $0$\\
$e_5$ & $0$ & $0$ & $0$ &$0$ & $0$
\end{tabular}
\end{center}
where $\alpha \colon A \to A$ is given by
\begin{align*}
\alpha(e_1) &= \nu e_1,\quad \alpha(e_2) = \nu^{-1}e_2,\quad \alpha(e_3) = \lambda e_3,\\
\alpha(e_4) &= e_4,\quad \alpha(e_5) = e_5.
\end{align*}

D. Yau, in \cite{yau}, proved that $(A,\cdot,\a)$ is Hom-flexible and Hom-Malcev admissible algebra.  In addition, by a direct calculation, we can easily verify that it satisfies condition \eqref{admissible JMP alg cond}. Hence, using Theorem \ref{admissible JMP alg}, we conclude that $(A,\cdot,\a)$ is an admissible Hom-JMP algebra.

\end{exa}

 Now we prove  that multiplicative admissible Hom-JMP algebras are power Hom-associative. The power associativity of admissible JMP algebras is  shown  in \cite{benayadi} (Proposition 2.2).
Let us begin by  recalling  the definition of a power Hom-associative algebra.
%%%%%%%%%%%%%%%
\begin{df}\cite{yau15}
\label{def:hompower}
Let $(A,\cdot,\alpha)$ be a Hom-algebra, $x \in A$, and $n$ be a positive integer.
\begin{enumerate}
\item
The $n$th Hom-power $x^n \in A$ is defined by
\begin{equation}
\label{hompower}
x^1 = x, \qquad
x^n = x^{n-1}\cdot\alpha^{n-2}(x),\ \ n \geq 2.
\end{equation}

\item
$A$ is called $n$th power Hom-associative if
\begin{equation}
\label{nhpa}
x^n =  \alpha^{n-i-1}(x^i)\cdot \alpha^{i-1}(x^{n-i}),
\end{equation}
for all $x \in A$ and $i \in \{1,\ldots, n-1\}$.
\item
$A$ is called power Hom-associative if $A$ is $n$th power Hom-associative for all $n \geq 2$.
\end{enumerate}
\end{df}
If the twisting map  $\a$ is the identity map, then the $n$th power Hom-associativity becomes
\begin{equation}
\label{npa}
x^n = x^i\cdot x^{n-i}.
\end{equation}
%%%%%%%%%%%%%%%%%%%%
 The class of power Hom-associative algebras contains  multiplicative right Hom-alternative algebras and non-commutative Hom-Jordan algebras.  Other results for power Hom-associative algebras can be found in \cite{yau15}.

  A well-known result of Albert \cite{albert1} says that an algebra $(A,\cdot)$ is power associative if and only if it is third and fourth power associative, i.e., the condition \eqref{npa} holds for $n = 3$, $4$.  Moreover, for \eqref{npa} to hold for $n = 3$, $4$, it is necessary and sufficient that
\[
(x\cdot x)\cdot x = x\cdot (x\cdot x) \quad\text{and}\quad ((x\cdot x)\cdot x)\cdot x = (x\cdot x)\cdot(x\cdot x),\; \textrm{for\; all}\; x \in A.
\]
 The Hom-versions of these statements are also proved in \cite{yau15}.  More precisely,  a multiplicative Hom-algebra $(A,\cdot,\alpha)$ is power Hom-associative if and only if it is third and fourth power Hom-associative, which in turn is equivalent to
\begin{equation}
\label{34}
x^3=x^2\cdot\alpha(x) = \alpha(x)\cdot x^2 \quad\text{and}\quad
x^4 =x^3\cdot\a^2(x)= \alpha(x^2)\cdot\alpha(x^2)
\end{equation}
for all $x \in A$.

The following result is the Hom-version of Proposition 2.2 in \cite{benayadi}.
%%%%%%%%%%%%%%%%%%
\begin{thm}
\label{thm:power}
Every multiplicative admissible Hom-MJP algebra is power Hom-associative.
\end{thm}
%%%%%%%%%%%%%%%%%%

\begin{proof}
As discussed above, by a result in \cite{yau15},  it is sufficient  to prove the two identities in \eqref{34}.  The Hom-flexibility  implies that
$$0 = as_A(x,x,x) = x^2\cdot\alpha(x) - \alpha(x)\cdot x^2,$$
which proves the first identity in \eqref{34}.  To prove the second  equality in \eqref{34}, since $(A,\circ,\a)$ is Hom-Jordan and due to Corollary \ref{equality of associator}, we get
\begin{align*}
0 &= as_A(x^2,\a(x),\a(x))\\
&= (x^2\cdot\a(x))\cdot\a^2(x)-\a(x^2)\cdot(\a(x)\cdot\a(x))\\
&=x^3\cdot\a^2(x)-\a(x^2)\cdot\a(x^2)
\end{align*}
We have proved the second identity in \eqref{34}.
\end{proof}

%%%%%%%%%%%%%%%%%%%%%%%%%%%%%%%%%%%%%%%%%%%%%%
\section{Pseudo-Euclidian Hom-JMP algebras}
%%%%%%%%%%%%%%%%%%%%%%%%%%%%%%%%%%%%%%%%%%%%%%
In this section, we extend the notion of pseudo-Euclidian JMP algebra to Hom-JMP algebras and provide some properties.
Let $(A,\cdot,\a)$ be a  Hom-algebra and $B: A \times A \to \mathbb{K}$  be a bilinear form. $B$ is called:
\begin{itemize}
  \item [(i)]  symmetric if $B(x,y)=B(y,x)$, $\forall x,y \in A$.
  \item [(ii)] non-degenerate if $B(x,y)=0,\ \forall y \in A \Rightarrow x=0$ and if $B(x,y)=0, \forall x \in A \Rightarrow y=0$.
  \item [(iii)] invariant if $B(x\cdot y,z)=B(x,y\cdot z)$, $\forall x,y,z \in A$.
\end{itemize}
 In this case, $(A,\cdot,\a,B)$ will be called a pseudo-Euclidian Hom-algebra if, in addition:
\begin{align*}
    & B(\a(x),y)=B(x,\a(y)),\ \quad \forall x,y \in A.
\end{align*}
\begin{df}
Let $(A,\{,\},\circ,\a)$ be a Hom-JMP algebra and $B:A \times A \to \mathbb{K}$ be a symmetric, non-degenerate and invariant bilinear form on $A$. We say that $(A,\{,\},\circ, \a,B)$ is a pseudo-Euclidian Hom-JMP algebra if $(A,\{,\},\a,B)$ and $(A,\circ,\a,B)$ are pseudo-Euclidian Hom-algebras.
\end{df}

\begin{df}
A Hom-JMP algebra $(A,\{,\},\circ,\a)$  is called Hom-pseudo-Euclidian if there exists $(B,\gamma)$, where $B$ is a symmetric and non-degenerate bilinear form on $A$ and $\gamma: A \to A$ is an homomorphism such that:
{\small\begin{align}\label{hom quadratic condition}
& B(\a(x),y)=B(x,\a(y));\  B(\{x,y\},\gamma(z))=B(\gamma(x),\{y,z\});\  B(x \circ y,\gamma(z))=B(\gamma(x), y \circ z),
\end{align}}
for all $x,y,z \in A$.
\end{df}
\begin{rem}
Note that we recover pseudo-Euclidian Hom-JMP algebras when $\gamma=id$.
\end{rem}
Let $(A,\{,\},\circ,B)$  be a pseudo-Euclidian JMP algebra.  We denote by $Aut_S(A,B)$ the set of symmetric automorphisms of $A$ with respect to $B$, that is automorphisms $\b:A \to A$ such that $B(\b(x),y)=B(x,\b(y)),\ \forall x,y \in A$.
 \begin{prop}\label{twistPseudEuc}
 Let $(A,\{,\},\circ,B)$  be a pseudo-Euclidian JMP algebra and $\a \in Aut_S(A,B)$. Then $A_\a=(A,\{,\}_\a,\circ_\a,\a,B_\a)$ is a pseudo-Euclidian Hom-JMP algebra, where for any $x,y \in A$
 \begin{align}\label{twist bilinear form}
 & \{x,y\}_\a=\{\a(x),\a(y)\};\ x \circ_\a y=\a(x) \circ \a(y);\ \textrm{and}\  B_\a(x,y)=B(\a(x),y).
 \end{align}

 \end{prop}
\begin{proof}
By Corollary \ref{coro twist}, $(A,\{,\}_\a,\circ_\a,\a)$ is a Hom-JMP algebra.

The linear form $B_\a$ is nondegenerate since $B$ is nondegenerate and $\a$ bijective. Now, let $x,y,z \in A$, then
\begin{align*}
 B_\a(\{x,y\}_\a,z)&=B(\a(\{\a(x),\a(y)\}),z)\\
&= B(\{\a(x),\a(y)\},\a(z))\\
&=B(\a(x),\{\a(y),\a(z)\})\\
&=B_\a(x,\{y,z\}_\a),
\end{align*}
and
\begin{align*}
B_\a(x \circ_\a y, z) &= B(\a(\a(x) \circ \a(y)),z)\\
&=B(\a(x) \circ \a(y),\a(z))=B(\a(x),\a(y) \circ \a(z))\\
&=B_\a(x, y \circ_\a z).
\end{align*}
On the other hand,
\begin{align*}
& B_\a(x,y)=B(\a(x),y)=B(x,\a(y))=B(\a(y),x)=B_\a(y,x),
\end{align*}
and
\begin{align*}
&B_\a(\a(x),y)=B(\a(\a(x)),y)=B(\a(x),\a(y))=B_\a(x,\a(y)).
\end{align*}

\end{proof}
The following result allows to obtain new pseudo-Euclidian Hom-JMP algebras starting from  multiplicative pseudo-Euclidian Hom-JMP algebras.

\begin{prop}
Let $(A,\{,\},\circ,\a,B)$ be a pseudo-Euclidian Hom-JMP algebra. For any $n \geq 0$, the quadruple
\begin{align}
& A_n=\Big(A,\{,\}_n=\a^n\{,\},\circ_n=\a^n\;  \circ, \a^{n+1},B_n\Big),
\end{align}
where $B_n$ is defined for $x,y \in A$ by $B_n(x,y)=B(\a^n(x),y)$, determines a pseudo-Euclidian Hom-JMP algebra.
\end{prop}
\begin{proof}
straightforward.
\end{proof}
We provide here a construction of pseudo-Euclidian Hom-JMP algebra from an arbitrary Hom-JMP algebra (not necessarily pseudo-Euclidian) .

Let $(A,\{,\},\circ)$ be a JMP-algebra and $A^*$ be the dual vector space of the underlying
vector space of $A$.  On the vector space $P=A\oplus A^*$, we define the following bracket $\{,\}_P$ and multiplication $\circ_P$ by:
\begin{align}\label{T extension}
& \{x+f,y+g\}_P:=\{x,y\}+f  ad_y-g   ad_x,\\
& (x+f)\circ_P (y+g):= x \circ y+ f  L_y + g  L_x, \ \forall (x,f),(y,g) \in P,
\end{align}
where $ad_x(y)=\{x,y\}$ and $L_x(y)=x\circ y$.
Moreover, we consider the bilinear form $B$ defined on $P$ by:
\begin{align}\label{bilinear form direct sum}
& B(x+f,y+g)=f(y)+g(x),\ \forall (x,f),(y,g) \in P.
\end{align}
$(P,\{,\}_P,\circ_P,B)$ is pseudo-Euclidian JMP algebra algebras called the  $T^*$-extension of $A$ by means of $A^*$.
\begin{prop}
  Let $(A, \{ , \},\circ)$ be a JMP algebra and $\alpha\in Aut(A)$. Then the endomorphism $\beta=\alpha+^t\alpha$ of
$P$ is an automorphism of $P$  if and only if $Im (\alpha^2-Id)\subseteq Z_J(A) \cap Z_M(A)$, where
$Z_J(A)$ is the center of $(A,\circ)$ and $Z_M(A)$ is the center of $(A, \{ , \})$ .
Hence, if $Im (\alpha^2-Id)\subseteq Z_J(A)\cap Z_M(A)$ then $(P, \{ , \}_{P,\beta},\circ_{P,\beta},B_\beta)$ is a regular pseudo-Euclidian Hom-JMP algebra.
\end{prop}
\begin{proof}Let $x,y\in A$ and $f,g\in A^*$.
\begin{align*}
\beta(\{x+f,y+g\}_{P})&=\beta(\{x,y\}+f ad_y-g  ad_x)\\
 &= \alpha(\{x,y\})+f ad_y \alpha -g ad_x  \alpha,
\end{align*}
and
\begin{eqnarray*}
\{\beta (x+f),\beta (y+g)\}_P&=&\{ \alpha(x)+f \alpha, \alpha(y)+g \alpha\}_P\\
\ &=& \{\alpha (x),\alpha (y)\}+f \alpha  ad_{ \alpha (y)} -g  \alpha ad_{ \alpha (x)} ,
\end{eqnarray*}
Then $\beta(\{x+f,y+g\}_P)=\{\beta (x+f),\beta (y+g)\}_P$ if and only if
$$\forall x,y\in A,\quad fad_y \alpha -g  ad_x  \alpha=f\alpha  ad_{ \alpha (y)} -g \alpha ad_{ \alpha (x)}.
$$
That is for all $z\in A$
$$f(\{y,\alpha (z)\})-g(\{x,\alpha (z)\})=f(\alpha\{\alpha(y),z\})-g(\alpha\{\alpha(x),z\}).$$
Hence, $\beta$ is an automorphism of $(P,\{,\}_P)$ if and only if $f(\{x,\alpha(y)\})=f(\alpha\{\alpha(x),y\})$, $\forall f\in A^*$ $\forall x,y \in A,$
  which is equivalent to $\{x,\alpha(y)\}=\alpha(\{\alpha(x),y\})$ $\forall x,y \in A.$

As a consequence, $\beta$ is an automorphism of $(P,\{,\}_P)$ if and only if $\{\alpha^2(x)-x,\alpha (y)\}=0$ $\forall x,y \in A$, ie. $Im(\alpha^2 -id)\subseteq Z_M(A)$, since $\alpha \in Aut(A).$

Similarly,
 $\beta$ is an automorphism of $(P,\circ_P)$ if and only if  $Im(\alpha^2 -id)\subseteq Z_J(A)$. Then $\beta \in Aut(P)$ if and only if $Im(\alpha^2 -id)\subseteq Z_M(A)\cap Z_J(A)$.

In the following we show that $\beta$ is symmetric with respect to $B$. Indeed, let $x,y\in A$ and $f,g\in A^*$
\begin{align*}
B(\beta (x+f),y+g)&=B(\alpha(x)+f \alpha,y+g)\\
&= f (\alpha(y))+g(\alpha(x))\\
&= B(x+f,\alpha(y)+g\alpha)
\\
&= B(x+f,\beta (y+g)).
\end{align*}

The last assertion  is a consequence of the previous calculations and  Proposition \ref{twistPseudEuc}.

\end{proof}
\begin{cor}
  Let $(A, \{ , \},\circ)$ be a JMP algebra and $\theta\in Aut(A)$ such that $\theta^2=id$ $(\theta$ is an involution$)$. Then   $(P, \{ , \}_{P,\beta},\circ_{P,\beta},B_\beta)$ is a regular pseudo-Euclidian Hom-JMP algebra, where  $\beta=\theta+^t\theta$.
\end{cor}

%%%%%%%%%%%%%%%%%%%%%%%%%%%%%%%%%%%%%%%%%%%%%%
\section{Hom-Lie-Jordan-Poisson Triple System}
%%%%%%%%%%%%%%%%%%%%%%%%%%%%%%%%%%%%%%%%%%%%%%
In this section, we generalize the notion of Lie-Jordan-Poisson triple system introduced in \cite{benayadi} to the Hom setting. We provide the relationships of this class of algebras with admissible Hom-JMP algebras.  Finally, we endow it with a symmetric non-degenerate invariant bilinear form and give some key constructions.

\begin{df}
\label{def:hjts}
A Hom-Lie triple system is a Hom-triple system $(L,[,,],\alpha=(\a_1,\a_2))$ that satisfies
the following conditions
\begin{itemize}
\item[(i)] $[x,y,z] = -[y,x,z]$ \quad\text{(left skewsymmetry)},
\item
[(ii)]
$[x,y,z] + [y,z,x] + [z,x,y]=0$ \quad\text{(ternary Jacobi identity)},
\item
[(iii)]
$
[\alpha_1(x),\alpha_2(y),[u,v,w]] = [[x,y,u],\alpha_1(v),\alpha_2(w)] + [\alpha_1(u),[x,y,v],\alpha_2(w)] + [\alpha_1(u),\alpha_2(v),[x,y,w]]
$,
\end{itemize}
 for all $u,v,w,x,y ,z\in L$.
\end{df}
A particular situation , interesting for our setting, occurs when  twisting maps $\a_i$ are all equal, that is $\a_1=\a_2=\a$, and $\a([x,y,z])=[\a(x),\a(y),\a(z)]$, for all $x,y,z \in L$. The Hom-Lie triple system $(L,[.,.,.],\a)$ is said to be multiplicative.
\begin{df}
A  Hom-Lie-Jordan-Poisson  triple system is a quadruple $(A,\{,,\},\circ, \a)$ such that
\begin{itemize}
  \item [(i)] $(A,\circ, \a)$ is a Hom Jordan algebra.
  \item [(ii)] $(A,\{ , , \},\a)$ is a multiplicative Hom-Lie triple system.
  \item [(iii)] $\{\a(x),\a(y),z \circ t \}= \{x,y,z\}\circ \a(t) + \a(z) \circ \{x,y,t\}$, $\forall x,y,z,t \in A$.
\end{itemize}
\end{df}
In the case where $(A,\circ, \a)$ is a  commutative Hom-associative algebra,  the quadruple $(A,\{,,\},\circ, \a)$ is called a Hom-Lie-Poisson triple system.

\begin{lem}\cite{NourouIssa}
Let $(A,\{,\},\a)$ be a Hom-Malcev algebra.  Then $(A,\{,,\},\a^2)$ is a multiplicative Hom-Lie triple system, where
\begin{align}
& \{x,y,z\}=2\{\{x,y\},\a(z)\}-\{\{y,z\},\a(x)\}-\{\{z,x\},\a(y)\},\ \forall x,y,z \in A.
\end{align}
\end{lem}

\begin{prop}
Let $(A,\{,\},\circ,\a)$ be a Hom-JMP algebra. Then the quadruple $(A,\{,,\},\circ_\a, \a^2)$  is a  Hom-Lie-Jordan-Poisson  triple system, where
\begin{align}\label{malcev==>LTS}
& \{x,y,z\}=2\{\{x,y\},\a(z)\}-\{\{y,z\},\a(x)\}-\{\{z,x\},\a(y)\},\ \forall x,y,z \in A,\\
&\label{Jordan==>HomJordan} x \circ_\a y=\a(x) \circ \a(y),\ \forall x,y \in A.
\end{align}
\end{prop}
\begin{proof}
  Let $x,y,z,t \in A$, we have
  \begin{align*}
    \{\a^2(x),\a^2(y),z \circ_\a t \}= & 2\{\{\a^2(x),\a^2(y)\},\a(z \circ_\a t)\}-\{\{\a^2(y),z \circ_\a t\},\a^3(x)\}-\{\{z \circ_\a t,\a^2(x)\},\a^3(y)\}\\
     =&\a\Big( 2\{\{\a(x),\a(y)\},\a(z) \circ \a(t)\}-\{\{\a(y),z \circ t\},\a^2(x)\}-\{\{z \circ t,\a(x)\},\a^2(y)\}\Big) \\
    =&\a\Big( 2\{\{x,y\},\a(z)\} \circ \a^2(t)+2\a^2(z) \circ \{\{x,y\},\a(t)\}-
    \{\{y,z\} \circ \a(t),\a^2(x)\}\\&-\{\a(z) \circ \{y,t\},\a^2(x)\}
    -\{\a(z) \circ \{t,x\},\a^2(y)\}-\{\{z ,x\}\circ \a(t),\a^2(y)\}\Big)\\
    =&\a\Big( 2\{\{x,y\},\a(z)\} \circ \a^2(t)+2\a^2(z) \circ \{\{x,y\},\a(t)\}-
\{\a(y),\a(z)\} \circ \{\a(t),\a(x)\}\\&-\{\{y,z\} ,\a(x)\}\circ \a^2(t)-\{\a(z) ,\a(x)\}\circ \{\a(y),\a(t)\}-\a^2(z) \circ \{\{y,t\},\a(x)\}
    \\&-\{\a(z),\a(y)\} \circ \{\a(t),\a(x)\}-\a^2(z) \circ \{\{t,x\},\a(y)\}\\&-\{\a(z) ,\a(x)\}\circ \{ \a(t),\a(y)\}-\{\{z ,x\},\a(y)\}\circ \a^2(t)\Big)\\
    =&\a\Big(\a^2(z) \circ (\{\{x,y\},\a(t)\}-\{\{y,t\},\a(x)\}- \{\{t,x\},\a(y)\}) \\&
    +(2\{\{x,y\},\a(z)\} -\{\{y,z\} ,\a(x)\}-\{\{z ,x\},\a(y)\}\circ) \a^2(t)\Big)
    \\=&\a^2(z) \circ_\a \{x,y,t\}
    +\{x,y,z\}\circ_\a \a^2(t).
  \end{align*}
\end{proof}
Let $(A,\{,,\},\a)$ be a multiplicative   Hom-Lie triple system and $B: A \times A \to \mathbb{K}$  be a symmetric non-degenerate  bilinear form. We say that $B$ is invariant if $$B(L(x,y)(z),t)=-B(z,L(x,y)(t)),\ \forall x,y,z,t \in A,$$
where $L(x,y)(z)=\{x,y,z\}$.
 In this case $(A,\{,,\},\a,B)$ will be called a pseudo-Euclidian Hom-Lie triple system if, in addition:
\begin{align*}
    & B(\a(x),y)=B(x,\a(y)),\ \quad \forall x,y \in A.
\end{align*}
\begin{df}
Let $(A,\{,,\},\circ,\a)$ be a Hom-Lie-Jordan-Poisson triple system  and $B:A \times A \to \mathbb{K}$ be a bilinear form on $A$. We say that $(A,\{,,\},\circ, \a,B)$ is a pseudo-Euclidian Hom-Lie-Jordan-Poisson triple system if $(A,\{,,\},\a,B)$ is a pseudo-Euclidian Hom-Lie triple system and $(A,\circ,\a,B)$ is a pseudo-Euclidian Hom-Jordan algebra.
\end{df}

\begin{df}
A Hom-Lie-Jordan-Poisson triple system  $(A,\{,,\},\circ,\a)$ is called Hom-pseudo-Euclidian if there exists $(B,\gamma)$, where $B$ is a symmetric and non-degenerate bilinear form on $A$ and $\gamma: A \to A$ is a homomorphism such that:
\begin{align*}
B(\a(x),y)&=B(x,\a(y)),\\ B(L(x,y)(z),\gamma(t))&=-B(\gamma(z),L(x,y)(t)),\\  B(x \circ y,\gamma(z))&=B(\gamma(x), y \circ z),
\end{align*}
for all $x,y,z,t \in A$.
\end{df}

\begin{cor}
Let $(A,\{,\},\circ,\a,B)$ be a pseudo-Euclidian Hom-JMP algebra. Then the $6$-uplet  $(A,\{,,\},\circ_\a, \a^2,B,\a)$  is a  Hom-pseudo-Euclidian Hom-Lie-Jordan-Poisson  triple system, where $\{,,\}$ and $\circ_\a$ are defined in \eqref{malcev==>LTS} and \eqref{Jordan==>HomJordan}.
\end{cor}
\begin{proof}
  Let $x,y,z,t\in A$, we have
  \begin{align*}
    B(\{x,y,z\},\a(t)) &= B(2\{\{x,y\},\a(z)\}-\{\{y,z\},\a(x)\}-\{\{z,x\},\a(y)\},\a(t))\\
     &= B(2\{\{x,y\},\a(z)\},\a(t))-B(\{\{y,z\},\a(x)\},\a(t))-B(\{\{z,x\},\a(y)\},\a(t)) \\
     & = -B(\a(z),2\{\{x,y\},\a(t)\})-B(\{y,z\},\{\a(x),\a(t)\})-B(\{z,x\},\{\a(y),\a(t)\})\\
     & = -B(\a(z),2\{\{x,y\},\a(t)\})-B(\{\a(y),\a(z)\},\{x,t\})-B(\{\a(z),\a(x)\},\{y,t\})\\
     & = -B(\a(z),2\{\{x,y\},\a(t)\})-B(\a(z),\{\{x,t\},\a(y)\})-B(\a(z),\{\{t,y\},\a(x)\})\\
     & = -B(\a(z),2\{\{x,y\},\a(t)\}-\{\{x,t\},\a(y)\}-\{\{t,y\},\a(x)\})\\
     &=-B(\a(z),\{x,y,t\}),
  \end{align*}
  which end proof.
\end{proof}

%%==============%%
%%              %%
%%  References  %%
%%              %%
%%==============%%

\end{document}